\documentclass{amsart}
\usepackage{amsmath}
\usepackage{amssymb}
\usepackage{graphicx}
\usepackage{pgf}
\usepackage{tikz}
\usetikzlibrary{arrows,automata}
\input xy
\xyoption{all}
\newtheorem{theorem}{Theorem}[section]

\newtheorem{question}{Question}
\newtheorem{lemma}[theorem]{Lemma}

\newtheorem{corollary}[theorem]{Corollary}
\newtheorem{proposition}[theorem]{Proposition}
\theoremstyle{definition}

\newtheorem{remark}{Remark}
\newenvironment{Proof}{{\textit{Proof}.}\ }{~$\square$\vspace{0.2truecm}}

\newcommand{\Cal}[1]{{\mathcal #1}}

\newcommand{\Core}{\mbox{\rm Core}}

\newcommand{\GL}{\mbox{\rm GL}}
\newcommand{\SL}{\mbox{\rm SL}}

\newcommand{\Z}{\mathbb{Z}}

\newcommand{\R}{\mathbb{R}}

\begin{document}
    \title[Intersection graphs]{Intersection graphs of  almost subnormal subgroups in general skew linear groups }
    
\author[Bui Xuan Hai]{Bui Xuan Hai$^{1,2}$}
\address{Bui Xuan Hai ({\tiny Corresponding author})\\[1] Faculty of Mathematics and Computer Science, University of Science; [2] Vietnam National University, Ho Chi Minh City, Vietnam.}
\email{bxhai@hcmus.edu.vn}

\author[Binh-Minh Bui-Xuan]{Binh-Minh Bui-Xuan$^{3}$}
\address{Binh-Minh Bui-Xuan\\ [3]Laboratoire d'Informatique de Paris 6 (LIP6), Centre National de la Recherche
	Scientifique (CNRS), Sorbonne Universite (SU UPMC)}
\email{buixuan@lip6.fr}

\author[Le Van Chua]{Le Van Chua$^{4,2}$}
\address{Le Van Chua\\[4] An Giang University; [2] Vietnam National University, Ho Chi Minh City, Vietnam.}
\email{lvchua.tag@moet.edu.vn}

\author[Mai Hoang Bien]{Mai Hoang Bien$^{1, 2}$}
\address{Mai Hoang Bien\\[1] Faculty of Mathematics and Computer Science, University of Science; [2] Vietnam National University, Ho Chi Minh City, Vietnam.}
\email{mhbien@hcmus.edu.vn}

\keywords{Intersection graphs; complete graphs; division rings; almost subnormal subgroups.  \\
\protect \indent 2010 {\it Mathematics Subject Classification.} 05C25, 16K20.}

 \maketitle

\begin{abstract} Let $D$ be a division ring, $n$  a positive integer, and $\GL_n(D)$ the general linear group of degree $n$ over $D$. In this paper, we study the induced subgraph of the intersection graph of $\GL_n(D)$ generated by all non-trivial proper almost subnormal subgroups of $\GL_n(D)$. We show that this subgraph is complete if it is non-null. This property will be used to study subgroup structure of a division ring. In particular, we prove that every non-central almost subnormal subgroup of the multiplicative group $D^*$ of a division ring $D$ contains a non-central subnormal subgroup of $D^*$.
\end{abstract}
\section{Introduction}

Let $D$ be a division ring, $n$ a positive integer, and $\GL_n(D)$  the general linear group of degree $n$ over $D$. Let $\Gamma(\GL_n(D))=(V, E)$ be the intersection graph of $\GL_n(D)$, where $V$ and $E$ denote the vertex set and the edge set respectively. Recall that 
$V$ consists of all non-trivial proper subgroups of $\GL_n(D)$. Two distinct vertices $A$ and $B$  are \textit{adjacent} if $A\cap B\neq  1$. The symbol $\{A, B\}$ denotes the edge between $A$ and $B$ when $A$ and $B$ are adjacent. Therefore, we have
$$V=\{A\mid A \mbox{ is  a non-trivial proper subgroup of }  \GL_n(D)\},$$ 
$$E=\{ \{A, B\} \mid A\ne B, A\cap B\ne 1\}.$$

Let $A_1, A_2, \dots, A_n\in V$.  If $A_i$ and $A_{i+1}$ are adjacent for any $1\le i\le n-1$, then we say that there is a \textit{path} from $A_1$ to $A_n$, and we denote this path by $(A_1, A_2, \dots, A_n)$.

Historically, the intersection graph of a group was firstly defined by B. Csakany and G. Pollak  \cite{PaCsPo_69} in 1969 with inspiration from the work \cite{Pa_Bo_64}. There are a lot of interesting results on intersection graphs of some classes of groups and its induced subgraphs (e.g., see \cite{Pa_AhTa_16, Pa_BiVi_20, Pa_DeRa_16, Pa_KaYa_15,  Pa_Ma_16, Pa_RaDe_16,Pa_Sh_10}). 

In this paper, we are interested in the study of properties of the induced subgraph $\Delta(\GL_n(D))$ of $\Gamma(\GL_n(D))$ generated by all non-central almost subnormal subgroups of $\GL_n(D)$, and their application in the study of the subgroup structure of $\GL_n(D)$. Observe that if $n=1$ and $D$ is a field, then $\Delta(\GL_n(D))$ is null, so in this paper, we always assume that $D$ is non-commutative in case $n=1$. We shall prove that $\Delta(\GL_n(D))$  is a clique of $\Gamma(\GL_n(D))$, that is, $\Delta(\GL_n(D))$ is complete, except the case when $n\ge 2$ but $D=F$ is a finite field. This result is very meaning in case $n=1$.  In this case, using the completeness of the subgraph $\Delta(\GL_1(D))$, we give the  affirmative answer to \cite[Question~ 2.7]{Pa_DeBiHa_19} concerning one problem on the subgroup structure of division rings. In fact, we show that every non-central almost subnormal subgroup of the multiplicative group $D^*$ of $D$ contains a non-central subnormal subgroup of $D^*$. This fact allows us in various cases to reduce the study of almost subnormal subgroups to the study of subnormal subgroups in division rings.

The content of the paper is as follows: In Section 2, we present some basic results on almost subnormal subgroups of $\GL_n(D)$ in order to study the induced subgraph $\Delta(\GL_n(D))$ of the intersection graph $\Gamma(\GL_n(D))$ of the group $\GL_n(D)$. Among results, we show that $\Delta(\GL_n(D))$ is complete (see Theorem~~\ref{th:2.8}). Section~ 3 spends for the case  $n=1$, where we investigate in details the structure of the subgraph $\Delta(\GL_1(D))$ and then using its properties to give the affirmative answer to \cite[Question~2.7]{Pa_DeBiHa_19} (see Theorem~\ref{th:3.3}). Finally, as an illustration, we show what from Theorem~\ref{th:3.3} we can get some results on almost subnormal subgroups in division rings using previous results on subnormal subgroups.

\section{The induced subgraph of $\Gamma(\GL_n(D))$ generated by\\ non-central almost subnormal subgroups}

Let $G$ be any group. Recall that a subgroup $N$ of $G$ is \textit{subnormal} in $G$ if there exists a sequence of subgroups $$N=N_r\triangleleft N_{r-1}\triangleleft N_{r-2}\triangleleft \dots \triangleleft N_1\triangleleft N_0=G,$$ 
where $N_{i+1}\triangleleft N_i$ for every $0\leq i<r$.

In accordance with Hartley \cite{Pa_Ha_89}, we say that a subgroup $N$ is \textit{almost subnormal} in $G$ if there exists a sequence of subgroups 
$$N=N_r< N_{r-1}< N_{r-2}< \dots < N_1< N_0=G$$ 
such that for every $0\le i<r$, either $N_{i+1}\triangleleft N_i$ or the index $[N_i: N_{i+1}]$ is finite. Such a sequence is called an \textit{almost normal series} in $G$. If no such a sequence of lesser length exists, then we say that $N$ is an almost subnormal subgroup of \textit{length} (or \textit{distance}) $n$.  Clearly, every subnormal subgroup in a given group is almost subnormal and the converse is not true. In this section, we consider almost subnormal subgroups in the general linear group $\GL_n(D)$ of degree $n$ over a division ring $D$. These subgroups were firstly studied in \cite{Pa_nbh_17}, and it was proved that if $D$ is infinite and $n\geq 2$ then every almost subnormal subgroup of $\GL_n(D)$ is normal \cite[Theorem~ 3.3]{Pa_nbh_17}. But this is not the case if $n=1$. Indeed, in  \cite[Section 2]{Pa_DeBiHa_19} and \cite{Pa_nbh_17}, it was shown that there exist infinitely many division rings whose multiplicative groups  contain almost subnormal subgroups that are not subnormal. Although for $n\ge 2$ and an infinite division ring $D$, in the group $\GL_n(D)$, every  almost subnormal subgroup is normal, we shall continue to use ``almost subnormal" instead of ``normal" to compare the results with the corresponding ones in the case $n=1$.

The aim of this section is to prove that the induced subgraph $\Delta(\GL_n(D))$ of the intersection graph $\Gamma(\GL_n(D))$ generated by all non-central almost subnormal subgroups of $\GL_n(D)$ is  a clique in $\Gamma(\GL_n(D))$. As an application, using this fact, we can get some property on the subgroup structure of $\GL_n(D)$. In particular, we can give the affirmative answer to \cite[Question~ 2.7]{Pa_DeBiHa_19}. The following lemma is obvious, so we omit its proof.
\begin{lemma}\label{lem:2.1} Assume that $f: A\longrightarrow B$ is a group epimorphism. 
	
	\begin{enumerate}
		\item If $N$ is an almost subnormal subgroup in $A$, then $f(N)$ is an almost subnormal subgroup in $B$.
		\item If $M$ is an almost subnormal subgroup in $B$, then $f^{-1}(M)$ is an almost subnormal subgroup in $A$. 
	\end{enumerate}
\end{lemma}

\begin{lemma}\label{lem:2.2}
	Let $G$ be a group and $H\le K\le G.$ If $M$ is a subgroup of $G$ and $[K:H]=n<\infty$, then $[K\cap M : H\cap M]\le n$.
\end{lemma}
\begin{Proof}
	Set $H_M=H\cap M$ and $K_M=K\cap M.$ Denoted by $\Cal L_M$ the set of all distinct left cosets of $H_M$ in $K_M$ and $\Cal L$ the set of all distinct left cosets of $H$ in $K.$ Consider the map $\Phi : \Cal L_M\rightarrow \Cal L$ defined by $\Phi (aH_M)=aH$ for $a\in K_M.$ It is easy to verify that $\Phi $ is well defined and injective. Hence, $|\Cal L_M|\leqslant |\Cal L|$ which implies that $[K\cap M : H\cap M]\le n$.
\end{Proof}
\begin{lemma}\label{lem:2.3}
	Let $G$ be a group. If $M$ and $N$ are almost subnormal subgroups in $G$, then $M\cap N$ is also an almost subnormal subgroup in $G.$
\end{lemma}
\begin{Proof}
Assume that 
$$N=N_n< N_{n-1}< \dots < N_1< N_0=G$$
is an almost normal series in $G$. Then, we have 
$$M\cap N=M\cap N_n< M\cap N_{n-1}< \dots < M\cap N_1< M\cap N_0=M.$$
For $0\le i<n$, if $N_{i+1}\triangleleft N_i,$ then $M\cap N_{i+1}\triangleleft M\cap N_i$.  If $[N_i: N_{i+1}]$ is finite, then $[M\cap N_i: M\cap N_{i+1}]$ is also finite by Lemma~\ref{lem:2.2}. Therefore, $M\cap N$ is  almost subnormal in $M$, and consequently,   $M\cap N$ is almost subnormal in $G$.
\end{Proof}

For further study, we need some results on generalized group identities. Let $G$ be any group with center $Z(G)$, and $\langle x_1,\dots, x_n\rangle$ be the free group generated by a set $\{x_1,\dots,x_n\}$ of non-commuting indeterminates. An element  $$w(x_1,\cdots,x_n)=a_1x_{i_1}^{m_1}a_2\cdots a_tx_{i_t}^{m_t}a_{t+1},$$ 
where $a_j\in G$ and $i_j\in \{1,\dots,n\}$,
in the free product $G* \langle x_1,\dots, x_n\rangle$,  is called a \textit{generalized group monomial} over $G$,  if for every $1\le j\le t-1$, the conditions $i_j=i_{j+1}$ and $m_{j}m_{j+1}<0$ imply that $a_{j+1}\not\in Z(G)$ (see \cite{Pa_GoMi_82,Pa_To_85}). 
Let $H$ be a subgroup of $G$.
We say that $w=1$ is a \textit{generalized group identity} of  $H$ or $H$ satisfies a generalized group identity $w=1$ over $G$ if $w(c_1,\cdots,c_n)=1$ for every $c_1,\cdots,c_n\in H$.  For some results on generalized group identities in skew general linear groups, we refer to \cite{Pa_GoMi_82} and \cite{Pa_To_85}. Useful results on generalized group identities of almost subnormal subgroups in division rings can be found in \cite{Pa_nbh_17}. In particular, the following result of \cite{Pa_nbh_17} will be used in the proof of the next theorem.

\begin{proposition}{\rm \cite[Theorem 2.2]{Pa_nbh_17}}\label{pro:2.4}
	Let $D$ be a division ring with infinite center $F$ and assume that $N$ is an almost subnormal subgroup in the multiplicative group $D^*$ of $D$. If $N$ satisfies a generalized group identity over $D^*$, then $N$ is central, that is, $N\subseteq F$.
\end{proposition}
\begin{lemma}\label{lem:ad} Let $D$ be a division ring with finite center. Then, every non-central almost subnormal subgroup of $D^*$ contains a non-central subnormal subgroup of $D^*$. 
\end{lemma}
\begin{proof} Assume that $N$ is a non-central almost subnormal subgroup of $D^*$. Then, by~\cite[Proposition 3.1]{Pa_Ca_01}, $N$ contains a subgroup $H$ of finite index such that $H$ is subnormal in $D^*$. We claim that $H$ is non-central. Indeed, if $H$ is central, then $a^n\in H\subseteq F$ for every $a\in N$, where $n=[N:H]$.  Since $F$ is finite, it follows that $a$ is a periodic element. If $a$ is non-central, then by~\cite[Proposition~ 2.2]{Pa_BiDu_14}, there exists a division subring $D_1$ of $D$ such that $a\in D_1$ and $D_1$ is centrally finite. Clearly, $N_1=N\cap D_1^*$ is non-central almost subnormal subgroup of $D_1^*$ (see Lemma~\ref{lem:2.3}). According to~\cite[Theorem 4.2]{Pa_nbh_17}, $N_1$ contains a non-cyclic free subgroup, a contradiction. Hence, $a\in F$ and so $N$ is central, again a contradiction. Consequently, $H$ is a non-central subnormal subgroup of $D^*$. 
\end{proof} 
\begin{theorem}\label{th:2.5} Let $D$ be a division ring and $M, N$ be two subgroups of $D^*$. If $M$ and $N$ are non-central almost subnormal in $D^*$, then so is $M\cap N$. 
\end{theorem}
\begin{Proof}
Let $F$ be the center of $D,$ and assume that $M$ and $N$ are non-central almost subnormal subgroups in $D^*$. If $F$ is finite, then the conclusion follows from Lemma \ref{lem:ad} and \cite[Theorem 5]{Pa_stuth_64}. Hence, we can assume that $F$ is infinite. By Lemma~\ref{lem:2.3}, $M\cap N$ is  almost subnormal in $D^*$. Deny the statement, assume that $M\cap N$ is central. Also, assume that $M$ and $N$ are two non-central almost subnormal subgroups in $D^*$ of lengths $s,r$ respectively such that $M\cap N\subseteq F$ and $r+s$ is minimal.  Let $$N=N_r < N_{r-1} < \cdots < N_1 < N_0=D^*,$$ and $$M=M_s < M_{s-1} < \cdots < M_1 < M_0=D^*$$ be almost normal series of $N$ and $M$ respectively. Because of the minimality of $r+s$ and in view of Lemma~\ref{lem:2.3}, $M\cap N_{r-1}$ and $N\cap M_{s-1}$ are two non-central almost subnormal subgroups in  $D^*$. There are two cases to examine.
	
		{\it Case 1. The case when $N$ is normal in $N_{r-1}$ and $M$ is normal in $M_{s-1}$.} Then $M\cap N$ is normal in $M\cap N_{r-1}$. Let $a\in (M\cap N_{r-1})\backslash F$. If $[a,x]=axa^{-1}x^{-1}\in F $ for any $x\in  N\cap M_{s-1}$, then $[[a,x],a]=1.$ Hence, $N\cap M_{s-1}$ satisfies a generalized group identity over $D^*$, so  $N\cap M_{s-1}$ is central by Proposition~\ref{pro:2.4}, which is a contradiction. Therefore, there exists $b\in N\cap M_{s-1}$ such that $[a,b]\not\in F.$ Since $[a,b]\in [M,M_{s-1}]\cap [N_{r-1},N]\subset M\cap N$, it follows that $M\cap N$ is non-central, which contradicts the assumption.

	{\it Case 2. The case when either $[N_{r-1}:N]$ or $[M_{s-1}:M]$ is finite}. Since the roles of $M$ and $N$ are similar, without loss of generality, we assume that $[N_{r-1}: N]<\infty$. Then $[M\cap N_{r-1}: M\cap N]$ is also finite by Lemma~\ref{lem:2.2}. Let $[M\cap N_{r-1}: M\cap N]=n.$ Then, $b^{n!}\in M\cap N\subseteq F$ for any $b\in M\cap N_{r-1}$. If we take $a\in D\backslash F$, then $M\cap N_{r-1}$ satisfies the generalized group identity $x^{n!}a^{n!}x^{-n!}a^{-n!}=1$ over $D^*$. In view of Proposition~\ref{pro:2.4}, $M\cap N_{r-1}$ is central, a contradiction.
	

	We see that both cases lead us to a contradiction, so the proof of the theorem is now complete. 
	\end{Proof}
	
Recall that a graph is  \textit{complete} if any two its vertices are adjacent. Note that for a division ring $D$ we have $\GL_1(D)=D^*$. So, for short, we use the notation $\Delta(D^*)$ instead of $\Delta(\GL_1(D))$.
\begin{theorem}\label{th:2.9} Let $D$ be a non-commutative division ring. Then, the following statements hold:
	\begin{enumerate}
		\item The graph $\Delta(D^*)$ is complete.
		\item Let $D'=[D^*, D^*]$ be the derived subgroup of $D^*$. If $N$ is a vertex of $\Delta(D^*)$, then either  $N$ is normal in $D^*$ or $N\cap D'$ is a proper subgroup of  $D'$ which is itself a vertex of $\Delta(D^*)$. 
		\item The vertex set of $\Delta(D^*)$ contain no finite subgroups of $D^*$.
	\end{enumerate}
\end{theorem} 

\begin{proof}
	(1) follows immediately from Theorem \ref{th:2.5}.
	
	(2) We claim that $D'$ is non-central. Indeed, if $D'$ is central, then $D^*$ is solvable, and in view of Hua's well-known result \cite{Pa_hua_1950}, $D$ is commutative, a contradiction. Assume that $N$ is a vertex of $\Delta(D^*)$, that is, $N$ is an almost subnormal subgroup in $D^*$. If $N$ contains $D'$, then $N$ is normal in $D^*$. Otherwise, $N\cap D'$ is a proper subgroup of $D'$ which is a vertex of $\Delta(D^*)$ by 
	Theorem \ref{th:2.5}. 
	
	(3) follows from \cite[Lemma~ 5.1]{Pa_nbh_17}. 
\end{proof} 


Concerning the case of $\Delta(\GL_n(D))$ for $n\geq 2$, the result would be stronger. To see this, we borrow the following theorem from 
\cite{Pa_nbh_17}. 
\begin{theorem}\label{th:2.7} {\rm (\cite[Theorem 3.3]{Pa_nbh_17})} Let $D$ be an infinite division ring, and let $n\geq 2$. Assume that $N$ is a non-central subgroup of $\GL_n(D)$. Then, the following conditions are equivalent:
	\begin{enumerate}
	\item  $N$ is almost subnormal in $\GL_n(D)$.
	\item $N$ is subnormal in $\GL_n(D)$.
	\item $N$ is normal in $\GL_n(D)$.
	\item $N$ contains $\SL_n(D)$.
	\end{enumerate}
\end{theorem}
This theorem shows that if $D$ is an infinite division ring then for $n\geq 2$, every non-central almost subnormal subgroup of $\GL_n(D)$ is normal. Moreover, it contains the special linear group $\SL_n(D)$ which is itself a non-central normal subgroup of $\GL_n(D)$. Hence, the vertex set of $\Delta(\GL_n(D))$ consists of all non-central normal subgroups of $\GL_n(D)$, and $\Delta(\GL_n(D))$ is obviously complete. Also, the condition (4) in Theorem \ref{th:2.7} shows that the vertex set of $\Delta(\GL_n(D))$ contains no finite subgroups of $\GL_n(D)$.	We summarize this in the following theorem.
\begin{theorem}\label{th:2.8} Let $D$ be an infinite division ring and $n\geq 2$. Then, the graph $\Delta(\GL_n(D))$ is complete. Moreover, the vertex set of $\Delta(\GL_n(D))$ consists of all non-central normal subgroups of $\GL_n(D)$. Also, a subgroup $N$ of $\GL_n(D)$ is a vertex of $\Delta(\GL_n(D))$ iff $N$ contains $\SL_n(D)$.
\end{theorem}

Denote by $V(\Delta)$ the vertex set of $\Delta=\Delta(\GL_n(D))$. In view of Theorem \ref{th:2.7}, we have
	$$\SL_n(D)=\bigcap_{N\in V(\Delta)} N.$$ 

\begin{remark} Theorem~\ref{th:2.8} says that for $n\ge 2$,  $\Delta(\GL_n(D))$ is complete if $D$ is infinite. We note that if $D$ is finite then $\Delta(\GL_n(D))$ may not be complete. Indeed, if $D$ is finite then $D=F$ is a finite field by Wedderburn's little Theorem. In this case, $\GL_n(F)$ is finite, which implies that every subgroup of $\GL_n(F)$ is of finite index in $\GL_n(F)$. Hence, every subgroup of $\GL_n(F)$ is almost subnormal in $\GL_n(F)$. It implies that $\Delta(\GL_n(F))$ is the induced subgraph of $\Gamma (\GL_n(F))$ generated by all non-central proper subgroups of $\GL_n(F)$. For instance, if $F=\Z/2\Z$ is the field of two elements, then every non-trivial subgroup of $\GL_n(\Z/2\Z)$ is non-central. Therefore, $\Delta(\GL_n(\Z/2\Z))=\Gamma(\GL_n(\Z/2\Z))$. It is easy to see (or see \cite[Section 2]{Pa_BiVi_20}), $\Delta(\GL_n(\Z/2\Z))$ has $4$ vertices and no edge. Even, by using the idea of \cite[Theorem 2.3 (2)]{Pa_BiVi_20}, we show that $\Delta(\GL_n(F))$ is not complete. To see this,  let $c=I_n+e_{12}$ and $d=I_n+e_{21}$ be two elements in $\GL_n(F)$, where $e_{ij}$ is denoted by the matrix in $M_n(F)$ in which the $(i,j)$-entry is $1$ and the other entries are $0$. Then, $C=\langle c\rangle$ and $D=\langle d\rangle$ are non-central subgroups of $\GL_n(F)$, that is, $C, D$ are vertices in $\Delta(\GL_n(F))$. Observe that $C $ and $D$ are distinct and $C\cap D=\{I_n\}$, so $\Delta(\GL_n(F))$ is not complete. 
\end{remark}
\section{The case $n=1$}

In this section, we are interested in the application of results about intersection graphs obtained in the precedent section to the study of subgroup structure of division rings. As we can see in Theorem \ref{th:2.9} and Theorem \ref{th:2.8}, if $D$ is a non-commutative division ring then the graph $\Delta(\GL_n(D))$ is complete.  Moreover, for $n\ge 2$, any its vertex is a normal subgroup of $\GL_n(D)$.  Unfortunately, this is not the case for $n=1$. In \cite{Pa_DeBiHa_19} and \cite{Pa_nbh_17}, there are the examples of division rings whose multiplicative groups contain almost subnormal subgroups that are even not subnormal. In \cite{Pa_DeBiHa_19}, the authors asked whether in $D^*$ any non-central almost subnormal subgroup contains a non-central subnormal subgroup. In fact, they posed the following question. 
\begin{question}{\rm (\cite[Question 2.7]{Pa_DeBiHa_19})}\label{question 1} Let $D$ be a division ring and $N$ an almost subnormal subgroup of $D^*$. If $N$ is non-central, then is it true that $N$ contains a non-central subnormal subgroup of $D^*$?
\end{question}
The affirmative answer to Question \ref{question 1} would give very useful tool to solve some problem on subgroup structure of $D^*$.
The main aim of this section is to use the completeness of the graph $\Delta(D^*)$ to study Question~ 1, and the affirmative answer to this question will be given.  In the sequent, this fact would have a number of consequences describing subgroup structure of a division ring $D$ (some illustrative examples will be given in the next after Theorem \ref{th:3.3}).

Let $G$ be a group and $\Omega$ be a subgraph of  the intersection graph $\Gamma(G)$. Denote by $V(\Omega)$ and $E(\Omega)$ the vertex set and edge set of $\Omega$ respectively. Consider two vertices $A, B\in V(\Omega)$. If there exists in $\Omega$ a path
$$(A=A_1, A_2, \dots, A_n=B)\eqno (1)$$
such that $A_{i+1}\lhd A_i$ for $1\le i\le n-1$, then we say that (1) is a \textit{downward path} from $A$ to $B$. 

\begin{theorem}\label{th:3.1} Let $D$ be a division ring with infinite center $F$. Assume that $A$ and $B$ are two distinct vertices of the graph $\Delta(D^*)$ such that $B$ is a subgroup of $A$. Then, there is a vertex $C$ of $\Delta(D^*)$ such that $C$ is a subgroup of $B$ and there is a downward path in $\Delta(D^*)$ from $A$ to $C$.  
\end{theorem}
\begin{Proof}   Since $B$ is almost subnormal in $D^*$, it is also almost subnormal in $A$. Suppose that $$B=N_r< N_{r-1}< \cdots < N_1< N_0=A$$ is an almost normal series of length $r$. Observe that $n\geq 1$ because $B$ is a proper subgroup of $A$. We show by induction on $r$ that $B$ contains a vertex $C$ such that there is a downward path in  $\Delta(D^*)$ from $A$ to $C$. Assume that $r=1$, that is, either $B$ is normal in $A$ or $[A:B]<\infty$. If $B$ is normal in $A$, then $(A,B)$ is a downward path in  $\Delta(D^*)$ from $A$ to $B$, so we take $C=B$. If $[A:B]<\infty$, then $\Core_A(B)$ is  a normal subgroup of finite index in $A$, say  $[A:\Core_A(B)]=\ell<\infty$. If $\Core_A(B)$ is contained in $F$, then $x^\ell y^\ell x^{-\ell} y^{-\ell}=1$ is a group identity of $A$, hence $A\subseteq F$ by Proposition~\ref{pro:2.4}, which contradicts the fact that $A$ is a vertex of $\Delta(D^*)$. Therefore, $\Core_A(B)$ is a vertex of $\Delta(D^*)$,  and we take $C=\Core_A(B)$. Now, assume that $r>1$ and  there exists a downward path in $\Delta(D^*)$  from $A$ to some vertex  $M_{r-1}$ which is a subgroup of $N_{r-1}$, say $(A, \dots, M_{r-1})$. Setting $M_r=M_{r-1}\cap N_r$, by Theorem~\ref{th:2.5}, $M_r=M_{r-1}\cap N_r$ is a non-central almost subnormal subgroup in $D^*$, that is, $M_r$ is a vertex of $\Delta(D^*)$. We claim that $M_r$ contains a subgroup $C$ such that $C$ is a vertex in $\Delta(D^*)$ and  $(A, \dots, M_{r-1}, C)$ is a downward path in $\Delta(D^*)$ from $A$ to $C$. Indeed, if $N_r$ is normal in $N_{r-1}$, then $M_{r}$ is normal in $M_{r-1}\cap N_{r-1}=M_{r-1}$, and it suffices to take $C=M_{r-1}$.  Now,  assume that $[N_{r-1}:N_r]=n<\infty$. Then,  $[M_{r-1}:M_r]=[N_{r-1}\cap M_{r-1}:N_r\cap M_{r-1}]\le [N_{r-1}:N_r]=n$ by Lemma~\ref{lem:2.2}, and it follows that  $\Core_{M_{r-1}}(M_r)$ is a normal subgroup of finite index in $M_{r-1}$, say $[M_{r-1}: \Core_{M_{r-1}}(M_r)]=m<\infty$. If $\Core_{M_{r-1}}(M_r)$ is central, then $M_{r-1}$ satisfies a group identity $x^{m}y^{m}x^{-m}y^{-m}=1$. Moreover, $M_{r-1}$ is almost subnormal in $D^*$ since $M_{r-1}$ is almost subnormal in $A$ and $A$ is almost subnormal in $D^*$. By Proposition~\ref{pro:2.4}, $M_{r-1}$ is central which is a contradiction. Hence, $\Core_{M_{r-1}}(M_r)$ is non-central, and we may take $C=\Core_{M_{r-1}}(M_r)$. The claim is shown, hence  the proof of the theorem is now complete.
\end{Proof}

From Theorem \ref{th:3.1}, it follows immediately the following corollary.

\begin{corollary}\label{cor:3.2} Let $D$ be a division ring with infinite center. Assume that $A$ is a non-central almost subnormal subgroup in $D^*$. If $A$ contains a proper non-central almost subnormal subgroup in $D^*$, then there exists a non-central subgroup $C$ of $D^*$ which is subnormal in $A$. 
\end{corollary}
Now, we are ready to give the affirmative answer to Question 1.
\begin{theorem}\label{th:3.3}
Let $D$ be a division ring. Then, every  non-central almost subnormal subgroup of $D^*$ contains a non-central subnormal subgroup of $D^*$.
\end{theorem}
\begin{Proof}
Assume that $N$ is a non-central almost subnormal subgroup in $D^*$ with an almost normal series 
$$N=N_r< N_{r-1}<\cdots < N_1< N_0=D^*.$$ 
Let $F$ be the center of $D$. If $F$ is finite, then by Lemma \ref{lem:ad}, $N$ contains a non-central subnormal subgroup of $D^*$. Hence, we can assume that $F$ is infinite. If $r=0$, then there is nothing to prove. Assume that $r\ge 1$. Clearly,  $ N_1,\dots, N_r$ are all vertices in $\Delta(D^*)$. 

\textit{Case 1: $N_1$ is normal in $D^*$.} 

If $r=1$, then $N=N_1$ is normal in $D^*$ and there is  nothing to do. Assume that $r>1$, so $N$ is a proper subgroup of $N_1$. Then, in view of 
Corollary~\ref{cor:3.2}, there exists a proper subgroup $M$ of $N$ which is subnormal in $N_1$. Consequently, $M$ is a non-central subnormal subgroup of $D^*$ which is contained in $N$.

\textit{Case 2:  $[D^*:N_1]<\infty$.}

Then, $\Core_{D^*}(N_1)$ is normal in $D^*$ of finite index, say $[D^*: \Core_{D^*}(N_1)]=m$. If $\Core_{D^*}(N_1)$ is central, then $x^my^mx^{-m}y^{-m}=1$ for any $x,y\in D^*$. Hence, in view of Proposition~\ref{pro:2.4}, $D^*$ is commutative, a contradiction. Therefore, $\Core_{D^*}(N_1)$ is a non-central subgroup, so it is  a vertex of $\Delta(D^*)$. If $\Core_{D^*}(N_1)\cap N=\Core_{D^*}(N_1)$, then $\Core_{D^*}(N_1)$ is a normal subgroup of $D^*$ which is contained in $N$. If  $\Core_{D^*}(N_1)\cap N<\Core_{D^*}(N_1)$, then, according to Corollary~\ref{cor:3.2}, there exists a non-central subnormal subgroup $M$ in  $\Core_{D^*}(N_1)$ such that $M$ is subnormal in $D^*$.

The proof of the theorem is now complete. 
\end{Proof}


For division rings, Theorem \ref{th:3.3} allows us to reduce various problems on almost subnormal subgroups to problems of subnormal subgroups. As a good illustration, we note the following problem. Let $D$ be a division ring with center $F$ and the multiplicative group $D^*$. A well-known result due to Hua \cite{Pa_hua_1950} states that if $D^*$ is solvable, then $D$ is a field. Later, in 1964, Stuth \cite{Pa_stuth_64} generalized Hua's result by proving that if a subnormal subgroup $N$ of $D^*$ is solvable, then $N$ must be contained in $F$. In \cite[Theorem 2.4]{Pa_HaTh_09}, it was proved that if $N$ is a locally solvable subnormal subgroup of $D^*$, then $N$ is central provided $D$ is algebraic over its center $F$. Later, it was conjectured \cite[Conjecture 1]{Pa_HaTh_13} that if a subnormal subgroup $N$ of $D^*$ is locally solvable, then $N$ must be central. From previous works we can see that this conjecture has the affirmative answer for the case when $N$ is  a locally nilpotent subgroup of $D^*$ (see \cite{Pa_huz_1960}) and for the case when $N$ is a locally solvable normal subgroup of $D^*$ (see \cite{Pa_Zales_1965}). In a recent work, Danh and Khanh completely solved this conjecture by proving the following theorem.

\begin{theorem}\label{dk}{\rm \cite[Theorem 2.5]{Pre_dk_20}} Let $D$ be a division ring with center $F$ and $N$ a subnormal subgroup of the multiplicative subgroup $D^*$ of $D$. If $N$ is locally solvable, then $N$ is central, that is, $N\subseteq F$. 
\end{theorem}
In view of Theorem \ref{th:3.3}, it is easy to carry over the results  for subnormal subgroups to the results for almost subnormal subgroups. For example, the following result is an immediate corollary of Theorem \ref{th:3.3} and Theorem \ref{dk}.

\begin{corollary}\label{cor:4.1}
	Let $D$ be a division ring and $N$  an almost subnormal subgroup of $D^*$. If $N$ is locally solvable, then $N$ is central. $\square$
\end{corollary}

By the same way, one can get several other results for almost subnormal subgroups in division rings. Here, we list only as examples two results. In both the two following corollaries, $D$ is a division ring with center $F$, and $N$ is an almost subnormal subgroup of the multiplicative group $D^*$ of $D$. 

\begin{corollary}\label{cor:4.10} Let $D$ be a division ring with uncountable center $F.$ If $N$ is radical over $F,$ then $N$ is central.
\end{corollary}
\begin{Proof} 
	By Theorem~\ref{th:3.3} and \cite[Theorem 2]{Pa_He_80}.
\end{Proof}

\begin{corollary}\label{cor:4.14} If $N$ is solvable-by-locally finite, then $N$ is central.
\end{corollary}
\begin{Proof} This is the consequence of Theorem~\ref{th:3.3} and \cite[Lemma 3.1]{Pa_HaTu_16}.
\end{Proof}

In the remain part, we focus the attention to the investigation of almost subnormal subgroups in the division ring of real quaternions. As an application of Theorem~ \ref{th:3.3}, we shall prove that in this division ring, a subgroup is almost subnormal iff it is normal. Some immediate corollaries from this fact will be also given. 

Let $\mathbb{H}=\R\oplus \R i \oplus \R j \oplus \R k$ be the division ring of real quaternions. For an element $\alpha=a+bi+cj+dk\in \mathbb{H},$ define $\overline\alpha=a-bi-cj-dk.$ The norm function $N: \mathbb{H}\longrightarrow \mathbb{R}^+$ is defined by  $$N(\alpha)=\alpha\overline\alpha=\overline\alpha\alpha=a^2+b^2+c^2+d^2\in \mathbb{R}^+, \mbox{ for any } \alpha\in \mathbb{H}.$$
We call $N(\alpha)$ the \textit{norm} of $\alpha$. Observe that
$N(\alpha\beta)=N(\alpha)N(\beta)$ for every $\alpha, \beta\in\mathbb{H}$, and it is easy to verify that $G_0=\{\alpha\in \mathbb{H}^*| N(\alpha)=1\}$ is  a non-central normal subgroup of $\mathbb{H}^*$.  
\begin{lemma}\label{lem:4.15}{\rm (See \cite[Example]{Pa_Gr_78})} 
	Let $G$ be a subgroup of $\mathbb{H}^*$.   Then, $G$ is a normal in $\mathbb{H}^*$ if and only if either $G$ is central or $G$ contains $G_0.$ 
\end{lemma}
Let $G$ be any group. We say that the \textit{subnormal property} holds for $G$ if every subnormal subgroup in $G$ is normal in $G$. Some authors call such a group $G$ a \textit{$T$-group}.
\begin{lemma}\label{lem:4.16} {\rm (See \cite[Example]{Pa_Gr_78})} 
	The subnormal property holds for $\mathbb{H}^*.$ In other phrase, $\mathbb{H}^*$ is a $T$-group.
\end{lemma}
\begin{theorem}\label{th:4.17} Let $\mathbb{H}$ be the division ring of real quaternions. Assume that $G$ is a non-central subgroup of the multiplicative group $\mathbb{H}^*$ of $\mathbb{H}$. Then, the following conditions are equivalent:
	\begin{enumerate}
		\item $G$ is almost subnormal in $\mathbb{H}^*.$
		\item $G$ is subnormal in $\mathbb{H}^*.$
		\item $G$ is normal in $\mathbb{H}^*.$
		\item $G$ contains $G_0.$
	\end{enumerate}
\end{theorem}
\begin{Proof} In view of Lemma \ref{lem:4.15}, the implication (4) $\Longrightarrow$ (3) holds. The implications (3) $\Longrightarrow$ (2) $\Longrightarrow$ (1) hold trivially. It remains to prove the implication (1) $\Longrightarrow$ (4). Thus, assume that  $G$ is an almost subnormal subgroup in $\mathbb{H}^*.$ By Theorem \ref{th:3.3}, $G$ contains a non-central subnormal subgroup $M$ of $\mathbb{H}^*.$ In view of Lemma \ref{lem:4.16}, $M$ is normal in $\mathbb{H}^*$, so by Lemma \ref{lem:4.15}, $M$ contains $G_0$. Therefore, $G$ contains $G_0.$
\end{Proof}
\begin{corollary}\label{cor:4.18} If $\mathbb{H}$ is the division ring of real quaternions, then  $$G_0=[\mathbb{H}^*,\mathbb{H}^*].$$
\end{corollary}
\begin{Proof} Clearly, $[\mathbb{H}^*,\mathbb{H}^*]\subseteq G_0$. In the other hand, since $[\mathbb{H}^*,\mathbb{H}^*]$ is a non-central normal subgroup of $\mathbb{H}$, in view of Theorem \ref{th:4.17}, $G_0\subseteq [\mathbb{H}^*, \mathbb{H}^*]$.
\end{Proof} 

\begin{corollary}\label{cor:4.19} The subgroup $G_0$ of $\mathbb{H}^*$ is perfect, and $[G_0,\mathbb{H}^*]=G_0$.
\end{corollary}
\begin{Proof} Observe that $[G_0, G_0]$ is a non-central subnormal subgroup of $\mathbb{H}^*$. In fact, if $[G_0, G_0]$ is central, then $G_0$ is solvable, and by \cite[Theorem 4]{Pa_stuth_64}, $G_0$ is central, a contradiction. Now, in view of Theorem \ref{th:4.17}, we have $[G_0, G_0]=G_0$. Consequently, $G_0\subseteq [G_0, \mathbb{H}^*]$. On the other hand, since $G_0$ is normal in $\mathbb{H}^*$, $[G_0, \mathbb{H}^*]\subseteq G_0$. Hence, $[G_0, \mathbb{H}^*]=G_0$. 
\end{Proof} 
\begin{corollary}\label{cor:4.20}  If a non-trivial subgroup $G$ of $\mathbb{H}^*$ is a perfect, then $G$ is non-central and contained in $G_0.$
\end{corollary}
\begin{Proof} Clearly, $G$ is non-central. Since   
	$N([\alpha,\beta])=1$ for every $\alpha,\beta\in G$,   $G=[G, G]$ is contained in $G_0.$ 
\end{Proof}
\begin{corollary}\label{cor:4.21} Assume that $G$ is a non-central normal subgroup of  $\mathbb{H}^*$. Then, $G$ is perfect if and only if $G$ is contained in $G_0.$
\end{corollary}
\begin{Proof}
	Assume that $G$ is contained in $G_0.$ Since $G$ is a non-central normal subgroup of $\mathbb{H}^*,$ $N$ contains $G_0$ by Theorem ~\ref{th:4.17}. Therefore, $G=G_0$ and so $G$ is perfect by Corollary~\ref{cor:4.19}. Conversely, suppose $G$ is perfect. According to Corollary~\ref{cor:4.20}, $G$ is contained in $G_0.$ 
\end{Proof}

\end{document}